\pdfoutput=1

\documentclass[11pt,leqno]{article}

\usepackage{amsmath,amsthm,amssymb}
\usepackage{graphicx,subfigure,booktabs,enumerate} 
\usepackage{mathabx}
\usepackage{color}

\let\originalincludegraphics\includegraphics
\renewcommand{\includegraphics}[2][]{%
 \IfFileExists{#2}{\originalincludegraphics[#1]{#2}}{%
  \fbox{\texttt{\detokenize{#2}}}}}

\theoremstyle{plain}
\newtheorem{thm}{Theorem}[section]

\newtheorem{lem}[thm]{Lemma}

\theoremstyle{remark}
\newtheorem{rem}[thm]{Remark}

\theoremstyle{definition}

\newtheorem{assum}[thm]{Assumption}

\makeatletter

\numberwithin{equation}{section}
\numberwithin{figure}{section}
\numberwithin{table}{section}

\title{\Large\bf Kernel-based collocation methods for Zakai equations}

\author{Yumiharu Nakano\\[1em]
        \small{Department of Mathematical and Computing Science, School of Computing} \\
        \small{Tokyo Institute of Technology} \\
        \small{W8-28, 2-12-1, Ookayama, Meguro-ku, Tokyo 152-8550, Japan} \\
		\small{e-mail: nakano@c.titech.ac.jp}
}

\date{\today}

\begin{document}

\maketitle

\begin{abstract}
We analyze kernel-based interpolation methods for numerical solutions of Zakai
equations in nonlinear filtering.  The stability of the spatial discretization
is quantified by row sums of the discrete differential operators.  The
resulting mean-square error estimate separates the one-step consistency error
from its finite-time amplification and yields convergence when these two
quantities balance.  A short finite-resolution example illustrates the
performance of a postprocessed interpolant.

\begin{flushleft}
\end{flushleft}
\begin{flushleft}
\end{flushleft}
\end{abstract}



\section{Introduction}\label{sec:1}
We are concerned with numerical methods for Zakai equations, 
linear stochastic partial differential equations of the form 
\begin{equation}
\label{eq:1.1}
\begin{split}
  du(t,x) = L_0u(t,x)dt  + \sum_{k=1}^mL_ku(t,x)dW_k(t),  
    \quad 0\le t\le T,   
\end{split}
\end{equation}
with initial condition $u(0,x)=u_0(x)$, 
where the process $\{W(t)=(W_1(t),\ldots,W_m(t))\}_{0\le t\le T}$ is an $m$-dimensional 
standard Wiener process on a complete probability space 
$(\Omega,\mathcal{F},\mathbb{P})$. 
Here, for each $k=0,1,\ldots, m$, the partial differential operator $L_k$ is given by 
\begin{align*}
 L_0f(x)&=\frac{1}{2}\sum_{i,j=1}^d\frac{\partial^2}{\partial x_i\partial x_j} 
 (a_{ij}(x)f(x)) 
  + \sum_{i=1}^d\frac{\partial}{\partial x_i}(b_i(x)f(x)), \\  
 L_{k}f(x) &= \beta_{k}(x)f(x) +\sum_{i=1}^d\frac{\partial}{\partial x_i} 
   (\gamma_{ik}(x)f(x)), \quad k=1,\ldots, m, 
\end{align*}
where $a=(a_{ij})$ is $\mathbb{R}^{d\times d}$-valued, 
$b=(b_i)$ is $\mathbb{R}^d$-valued, 
$\beta=(\beta_k)$ is $\mathbb{R}^{m}$-valued, 
$\gamma=(\gamma_{ik})$ is $\mathbb{R}^{d\times m}$-valued, 
and $u_0$ is $\mathbb{R}$-valued, 
all of which are defined on $\mathbb{R}^d$.
The conditions for these functions are described in Section \ref{sec:2} below.  

It is well known that solving a Zakai equation amounts to computing the optimal filter for 
a partially observed diffusion process.  
We refer to Rozovskii \cite{roz:1990}, Kunita \cite{kun:1990}, 
Liptser and Shiryaev \cite{lip-shi:2001}, Bensoussan \cite{ben:1992}, 
Bain and Crisan \cite{bai-cri:2009}, and the references therein 
for Zakai equations and their relation to nonlinear filtering. 
For linear diffusion models, the optimal filters often admit finite-dimensional
realizations; that is, they can be represented by finite-dimensional stochastic and
deterministic differential equations.  For nonlinear diffusion models, such realizations
are difficult to obtain except in special cases (see Bene{\v{s}} \cite{ben:1981} and
\cite{ben:1992}).  This motivates numerical approaches to Zakai equations.  Several
approximation methods have been developed over the past several decades, including the
finite difference method
(see Yoo \cite{yoo:1999}, Gy{\"o}ngy \cite{gyo:2014} and the references therein), 
the particle method (see Crisan et al.~\cite{cri-etal:1998}), 
a series expansion approach (Lototsky et al.~\cite{lot-etal:1997}), 
Galerkin-type approximation (Ahmed and Radaideh \cite{ahm-rad:1997} and 
Frey et al.~\cite{fre-etal:2013})
and the splitting-up method (Bensoussan et al.~\cite{ben-etal:1990}). 

In the present paper, we approximate $u(t,x)$ by a collocation method based on
kernel interpolation. Given a point set
$\Gamma=\{x_1,\ldots,x_N\}\subset\mathbb{R}^d$ and a positive definite function  
$\Phi:\mathbb{R}^d\to\mathbb{R}$, the function  
\[
 I(f)(x):=\sum_{j=1}^N(A^{-1}f|_{\Gamma})_j\Phi(x-x_j), \quad x\in\mathbb{R}^d, 
\]
interpolates $f$ on $\Gamma$. Here, $A=\{\Phi(x_j-x_{\ell})\}_{j,\ell=1,\ldots,N}$, 
$f|_{\Gamma}$ is the column vector composed of $f(x_j)$, 
$j=1,\ldots,N$, and 
$(A^{-1}z)_j$ denotes the $j$-th component of $A^{-1}z$ for $z\in\mathbb{R}^N$.  
Thus, with time grid $\{t_0,\ldots,t_n\}$, the function $u^h$ recursively defined by  
\begin{align*}
 u^h(t_i,x) &= u^h(t_{i-1},x)+L_0I(u^h(t_{i-1},\cdot))(x)(t_i-t_{i-1}) \\ 
  &\quad + \sum_{k=1}^mL_kI(u^h(t_{i-1},\cdot))(x)(W_k(t_i)-W_k(t_{i-1})), 
 \quad i=1,\ldots,n,\;\; x\in \mathbb{R}^d,  
\end{align*}
is a natural candidate for an approximate solution of \eqref{eq:1.1}. 
The approximation above can be seen as a kernel-based (or meshfree) collocation method for 
stochastic partial differential equations. 
Kansa \cite{kan:1990b} introduced meshfree collocation for deterministic partial
differential equations.  Since then, the method has been studied extensively through
numerical experiments and practical applications.  Regarding rigorous convergence,
Schaback \cite{sch:2010} and Nakano \cite{nak:2017} considered deterministic linear
operator equations and fully nonlinear parabolic equations, respectively.
Nevertheless, even for parabolic equations, few explicit combinations of point sets and
kernels are known to guarantee convergence.  One exception is the one-dimensional heat
equation considered by Hon et al.~\cite{hon-etal:2014}.  A main difficulty is controlling
repeated kernel interpolation.  A straightforward estimate of $|I(f)(x)|$ involves the
condition number of the matrix $A$,
which in general rapidly diverges to infinity (see Wendland \cite{wen:2005}). 
We quantify the finite-dimensional amplification by row sums of the
differentiated cardinal basis, without requiring them to remain uniformly
bounded.  Theorem \ref{thm:3.2} gives a mean-square error bound and convergence
when this amplification grows slowly enough to be compensated by the one-step
temporal and spatial consistency errors.

\paragraph{Relation to the published and earlier versions.}
The version published in \cite{nak:2019a} and the subsequent revision
\texttt{arXiv:1710.09090v8}
claimed convergence for concrete choices of a Wendland kernel and collocation
points.  The key step was a uniform bound on the absolute row sums of the
differentiated cardinal basis.  In the published version, Assumption 3.1(ii)
controlled the number of relatively large entries of the inverse interpolation
matrix, and Lemma 3.7 was intended to imply the required row-sum bound.  The
subsequent arXiv revision instead imposed a counting condition directly on the
quantities $|D^\alpha Q_j(x)|$, and its Lemma 3.5 was intended to give the same
uniform bound.

The proofs of these two stability lemmas contain the same unresolved step.  For
a fixed evaluation point $\tilde x$, a Kergin polynomial $p$ was constructed to
interpolate
$\operatorname{sgn}(D^\alpha Q_j(\tilde x))$ only at the collocation points
whose indices lie outside a set $\mathcal J(\tilde x)$ of ``small'' terms.  The
proof then chose an arbitrary nearest collocation point $x_j$ to a point $x$ and
used $|p(x_j)|\le1$.  The nearest point need not have its index outside
$\mathcal J(\tilde x)$, so this bound was not supplied by the interpolation
conditions.  Consequently, the asserted uniform bound for
$\sum_j|D^\alpha Q_j(x)|$ does not follow from the assumptions in either
version.  This is not repaired merely by prescribing values at all collocation
points, because the resulting polynomial degree would depend on $N$ and the
subsequent Bernstein estimate would lose the required uniformity.  Thus the
published Theorem 3.4 and the arXiv Theorem 3.3 were not established by the
proofs given there, and their convergence claims for the displayed kernel and
point-set sequences are withdrawn.

The present version replaces that step with a conditional error estimate.  The
computable row sums $M_{k,h}$ are retained explicitly in
\eqref{eq:3.stab}, and Lemma \ref{lem:3.7} gives the corresponding
finite-dimensional stability estimate without using Kergin interpolation or an
extension theorem.  Theorem \ref{thm:3.2} proves
\[
 \max_{i=1,\ldots,n}\max_{x\in\Gamma\cup\tilde\Gamma}
 \mathbb E|u(t_i,x)-u^h(t_i,x)|^2
 \le C\exp(T\kappa_h)
 \bigl(\Delta t+(\Delta x)^{2\tau-3}\bigr).
\]
Accordingly, convergence now requires the additional balance condition
\eqref{eq:3.balance}.  Neither Assumption \ref{assum:3.1} alone nor the
finite-resolution calculations in Section \ref{sec:4} establish this condition
for the displayed point-set sequence.  The corrected estimate is therefore a
conditional convergence result, not a recovery of the earlier concrete
convergence claim.

For completeness, this version also makes the following corrections and
clarifications.
\begin{enumerate}[\rm (a)]
\item Assumption \ref{assum:2.1} now includes
$u_0\in H^\nu(\mathbb R^d)$ for every $\nu\in\mathbb N$, as required by the
Sobolev estimates used in the proof.  The time argument in the stochastic term
of \eqref{eq:1.1} and the integration variables in the pointwise formulation
have also been corrected.
\item The common-null-set smooth modification in
Lemma \ref{lem:smooth-version} follows from degenerate Sobolev theory together
with Hilbert-space-valued stochastic integrals.  It does not invoke a
nondegenerate existence theorem as
if that theorem directly covered the semidefinite condition in Assumption
\ref{assum:2.1}(ii).
\item The native space is defined using the $L^2$ Fourier transform rather than
by an unnecessarily restrictive $L^1$ condition.  In Lemma \ref{lem:3.5}, the
zeros estimate is applied to $f-I(f)$, and the scaling argument is stated so
that its constant is independent of the expanding radius $R$.
\item The temporal regularity estimate uses $D^\alpha L_ku$ as the integrand
and treats the drift term by the Cauchy--Schwarz inequality rather than
It{\^o}'s isometry.  The nodal matrix formula is also written for
$L_kI(u^h)$, which is the quantity actually represented by $A_kA^{-1}$.
\item The exact solution and the reported error values in the numerical example
have been corrected.  Section \ref{sec:4} is now limited to a
finite-resolution performance illustration of a postprocessed interpolant and
is not presented as evidence for the asymptotic balance condition.
\end{enumerate}

The remainder of the paper is organized as follows.  Section \ref{sec:2}
introduces notation and recalls the results on Zakai equations and kernel
interpolation used below.  Section \ref{sec:3} presents the approximation method
and proves the conditional convergence theorem.  Section \ref{sec:4} reports a
finite-resolution performance example.

\section{Preliminaries}\label{sec:2}

\subsection{Notation}\label{sec:2.1}

Throughout this paper, 
we denote by $a^{\mathsf{T}}$ the transpose of a vector or matrix $a$. 
For $a=(a_i)\in\mathbb{R}^{\ell}$ 
we set $|a|=(\sum_{i=1}^{\ell}(a_{i})^2)^{1/2}$.  
For a multiindex $\alpha=(\alpha_1,\ldots,\alpha_d)$ of nonnegative integers, 
the differential operator $D^{\alpha}$ is defined as usual by  
\begin{equation*}
 D^{\alpha}=\frac{\partial^{|\alpha|_1}}
  {\partial x_1^{\alpha_1}\cdots \partial x_d^{\alpha_d}}  
\end{equation*}
with $|\alpha|_1=\alpha_1+\cdots +\alpha_d$. 
For an open set $\mathcal{O}\subset\mathbb{R}^d$, we denote 
by $C^{\kappa}(\mathcal{O})$ the space of real-valued functions 
on $\mathcal{O}$ with continuous derivatives up to order $\kappa\in\mathbb{N}$, 
with the norm 
\[
 \|f\|_{C^{\kappa}(\mathcal{O})}=\max_{|\alpha|_1\le\kappa}\sup_{x\in\mathcal{O}}
  |D^{\alpha}f(x)|. 
\]
Further, we denote by $C^{\infty}_0(\mathbb{R}^d)$ 
the space of infinitely differentiable functions on $\mathbb{R}^d$ with compact support. 
For any $p\in [1,\infty)$ and any open set $\mathcal{O}\subset\mathbb{R}^d$, 
we denote by $L^p(\mathcal{O})$ the space of all measurable functions 
$f:\mathcal{O}\to \mathbb{R}$ such that 
\[
 \|f\|_{L^p(\mathcal{O})}:=\left\{\int_{\mathcal{O}}|f(x)|^pdx\right\}^{1/p}<\infty. 
\]
For $\kappa\in\mathbb{N}$, 
we write $H^{\kappa}(\mathcal{O})$ for the space of all measurable functions 
$f$ on $\mathcal{O}$ such that 
the generalized derivatives $D^{\alpha}f$ exist for all $|\alpha|_1\le \kappa$ and 
that  
\begin{equation*}
 \|f\|_{H^{\kappa}(\mathcal{O})}^2 
 := \sum_{|\alpha|_1\le\kappa}\|D^{\alpha}f\|^2_{L^2(\mathcal{O})}<\infty. 
\end{equation*}
For later use, we also set
\begin{equation*}
 |f|_{H^{\kappa}(\mathcal{O})}^2
 :=\sum_{|\alpha|_1=\kappa}\|D^{\alpha}f\|^2_{L^2(\mathcal{O})},
 \qquad \kappa\in\mathbb N.
\end{equation*}
In addition, for $0< r<1$, 
we write $H^{\kappa+r}(\mathcal{O})$ for the space of all measurable functions 
$f$ on $\mathcal{O}$ such that 
the generalized derivatives $D^{\alpha}f$ exist for all $|\alpha|_1\le \kappa$ and 
that $\|f\|_{H^{\kappa+r}(\mathcal{O})}^2 
 := \|f\|_{H^{\kappa}(\mathcal{O})}^2 + |f|_{H^{\kappa+r}(\mathcal{O})}^2<\infty$ with  
\begin{equation*}
  |f|_{H^{\kappa +r}(\mathcal{O})}^2
  =\sum_{|\alpha|_1=\kappa}
  \int_{\mathcal{O}}\int_{\mathcal{O}}\frac{|D^{\alpha}f(x)-D^{\alpha}f(y)|^2}
 {|x-y|^{d+2r}}dxdy. 
\end{equation*}
For $x\in\mathbb{R}$ we use the notation $\lfloor x\rfloor=\max\{n\in\mathbb{Z}: n\le x\}$. 
By $C$ we denote positive constants that may vary from line to line 
and that are independent of $h$ introduced below.

\subsection{Zakai equations}\label{sec:2.2}

We impose the following conditions on the coefficients of the equation \eqref{eq:1.1}: 
\begin{assum}
\label{assum:2.1}
\begin{enumerate}[\rm (i)]
\item All components of $a$, $b$, $\beta$, and $\gamma$, as well as $u_0$,
are infinitely differentiable with bounded continuous derivatives of every order.
In addition, $u_0\in H^\nu(\mathbb{R}^d)$ for every $\nu\in\mathbb{N}$.
\item For any $x\in\mathbb{R}^d$,  
\begin{equation*}
 \xi^{\mathsf{T}}(a(x)-\gamma(x)\gamma(x)^{\mathsf{T}})\xi\ge 0, \quad \xi\in\mathbb{R}^d. 
\end{equation*}
\end{enumerate}
\end{assum}
Assumption \ref{assum:2.1} and
Gerencs{\'e}r et al.~\cite[Theorem 2.1]{ger-etal:2015} imply the existence of a
unique predictable process $\{u(t)\}_{0\le t\le T}$ with the following
properties:
\begin{enumerate}[\rm (i)] 
\item for every $\nu\in\mathbb{N}$, $u$ has a version with paths in
$C([0,T];H^{\nu}(\mathbb{R}^d))$; these versions may be chosen consistently
on a common set $\Omega_0\in\mathcal{F}$ with $\mathbb{P}(\Omega_0)=1$;
\item for $\varphi\in C^{\infty}_0(\mathbb{R}^d)$, 
\begin{equation}
\label{eq:2.1}
\begin{split}
  (u(t),\varphi) = (u_0,\varphi) + \int_0^t(u(s,\cdot),L_0^*\varphi)ds  
    + \sum_{k=1}^m\int_0^t(u(s,\cdot),L_k^*\varphi)dW_k(s),  
    \quad 0\le t\le T, \;\;\text{a.s.}  
\end{split}
\end{equation}
\end{enumerate}
Here,  $(\cdot,\cdot)$ denotes the inner product in $L^2(\mathbb{R}^d)$, and 
for each $k=0,1,\ldots, m$, the partial differential operator $L_{k}^*$ is the 
formal adjoint of $L_k$. 
Moreover, $u$ satisfies
\begin{equation*}
 \mathbb{E}\left[\sup_{0\le t\le T}\|u(t,\cdot)\|^2_{H^{\nu}(\mathbb{R}^d)}\right]
  \le C\|u_0\|^2_{H^{\nu}(\mathbb{R}^d)}, \quad \nu\in\mathbb{N}. 
\end{equation*}

The preceding degenerate Sobolev theory also supplies the regularity needed in
the modification argument of
\cite[Proposition 3, Section 1.3, Chapter 4]{roz:1990}.  For completeness, we
record the resulting statement and the point at which the degeneracy of
$a-\gamma\gamma^{\mathsf T}$ is harmless.
\begin{lem}
\label{lem:smooth-version}
Under Assumption \ref{assum:2.1}, there exist a version $\tilde u$ of $u$ and
a set $\Omega_*\in\mathcal F$, $\mathbb P(\Omega_*)=1$, such that
$\tilde u(t,\cdot,\omega)\in C^\infty(\mathbb R^d)$ for every
$(t,\omega)\in[0,T]\times\Omega_*$.  Moreover, for every multi-index
$\alpha$, the stochastic integral on the right-hand side below has a version jointly continuous in
$(t,x)$, and, on the same set $\Omega_*$,
\begin{equation}
\label{eq:2.2}
 D^{\alpha}\tilde{u}(t,x)=D^{\alpha}u_0(x)+\int_0^tD^{\alpha}L_0\tilde{u}(s,x)ds 
  + \sum_{k=1}^m\int_0^t 
 D^{\alpha}L_k\tilde{u}(s,x)dW_k(s)
\end{equation}
for all $(t,x)\in[0,T]\times\mathbb R^d$.
\end{lem}
\begin{proof}
Fix a multi-index $\alpha$ and an integer $r>d/2$.  Choose $\nu$ so large
that $\nu\ge r+|\alpha|_1+2$.  Since all derivatives of the coefficients are
bounded, the maps
\[
 f\longmapsto D^\alpha L_0f,\qquad
 f\longmapsto D^\alpha L_kf
\]
are bounded from $H^\nu(\mathbb R^d)$ to $H^r(\mathbb R^d)$, for
$k=1,\ldots,m$.  The Sobolev estimates above therefore imply
\[
 \mathbb E\int_0^T\|D^\alpha L_0u(s)\|_{H^r}^2ds<\infty,
 \qquad
 \mathbb E\int_0^T\|D^\alpha L_ku(s)\|_{H^r}^2ds<\infty.
\]
Consequently,
\[
 A^\alpha(t):=\int_0^tD^\alpha L_0u(s)ds,\qquad
 M^\alpha(t):=\sum_{k=1}^m\int_0^tD^\alpha L_ku(s)dW_k(s)
\]
have continuous $H^r(\mathbb R^d)$-valued versions.  Testing against a
countable dense subset of $C_0^\infty(\mathbb R^d)$ in \eqref{eq:2.1}, and
then using continuity in $H^r$, gives, outside one null set,
\[
 D^\alpha u(t)=D^\alpha u_0+A^\alpha(t)+M^\alpha(t),
 \qquad 0\le t\le T,
\]
as an identity in $H^r(\mathbb R^d)$.

Since $r>d/2$, the embedding $H^r(\mathbb R^d)\hookrightarrow
C_b(\mathbb R^d)$ is continuous.  Thus $A^\alpha$ and $M^\alpha$ have
versions continuous as maps from $[0,T]$ to $C_b(\mathbb R^d)$, hence
jointly continuous in $(t,x)$.  Moreover, point evaluation is a continuous
linear functional on $H^r(\mathbb R^d)$, so evaluating the $H^r$-valued
stochastic integral gives the corresponding scalar stochastic integral at
each $x$.  Taking continuous representatives in the
last identity proves \eqref{eq:2.2} simultaneously for all $(t,x)$.  Finally,
the collection of multi-indices is countable.  Intersecting the corresponding
probability-one sets, and using Sobolev embedding at arbitrarily high orders,
produces a single set $\Omega_*$ and a spatially $C^\infty$ version
$\tilde u$ with all the stated properties.  Notice that this modification
step uses only the Sobolev regularity and estimates furnished by
\cite[Theorem 2.1]{ger-etal:2015}; it does not require strict
super-parabolicity.
\end{proof}

In particular, taking $\alpha=0$ in Lemma \ref{lem:smooth-version} shows
that $\tilde u$ solves the Zakai equation in the strong pointwise sense on
the common probability-one set $\Omega_*$. 


\subsection{Kernel-based interpolation}\label{sec:2.3}

In this subsection, we recall the basic interpolation theory for
positive definite functions.
We refer to \cite{wen:2005} for a complete account. 
Let $\Phi: \mathbb{R}^d\to \mathbb{R}$ be a radial and positive definite function, 
i.e., $\Phi(\cdot)=\phi(|\cdot|)$ for some $\phi:[0,\infty)\to\mathbb{R}$ and 
for every $\ell\in\mathbb{N}$, every collection of pairwise distinct
$y_1,\ldots,y_{\ell}\in\mathbb{R}^{d}$, and every
$\alpha=(\alpha_i)\in\mathbb{R}^{\ell}\setminus\{0\}$, we have 
\begin{equation*}
  \sum_{i,j=1}^{\ell}\alpha_i\alpha_j\Phi(y_i-y_j)>0. 
\end{equation*}
Let $\Gamma=\{x_1,\cdots,x_N\}$ be a finite subset of $\mathbb{R}^d$ and 
put $A=\{\Phi(x_i-x_j)\}_{1\le i,j\le N}$. 
Then $A$ is invertible; hence, for any $g:\mathbb{R}^d\to\mathbb{R}$, the function
\begin{equation*}
I(g)(x) = \sum_{j=1}^N(A^{-1}(g|_{\Gamma}))_j\Phi(x-x_j), \quad x\in\mathbb{R}^d, 
\end{equation*}
interpolates $g$ on $\Gamma$. 
If $\Phi\in C(\mathbb{R}^d)\cap L^1(\mathbb{R}^d)$, then 
\begin{equation*}
 \mathcal{N}_{\Phi}(\mathbb{R}^d):=\left\{f\in C(\mathbb{R}^d)\cap L^2(\mathbb{R}^d): 
  \widehat{f}/\sqrt{\widehat{\Phi}}\in L^2(\mathbb{R}^d)\right\}, 
\end{equation*}
called the native space, is a real Hilbert space with the inner product 
\begin{equation*}
 (f,g)_{\mathcal{N}_{\Phi}(\mathbb{R}^d)}=(2\pi)^{-d/2}\int_{\mathbb{R}^d} 
 \frac{\widehat{f}(\xi)\overline{\widehat{g}(\xi)}}{\widehat{\Phi}(\xi)}d\xi 
\end{equation*}
and the norm 
$\|f\|_{\mathcal{N}_{\Phi}(\mathbb{R}^d)}^2:=(f,f)_{\mathcal{N}_{\Phi}(\mathbb{R}^d)}$. 
Here, $\widehat{f}$ denotes the $L^2$ Fourier transform, whose normalization is
initially given, for $f\in L^1(\mathbb{R}^d)\cap L^2(\mathbb{R}^d)$, by
\begin{equation*}
 \widehat{f}(\xi)=(2\pi)^{-d/2}\int_{\mathbb{R}^d}f(x)e^{-\sqrt{-1}x^{\mathsf{T}}\xi} dx, 
 \quad \xi\in\mathbb{R}^d. 
\end{equation*}
Moreover, $\mathbb{R}^d\times\mathbb{R}^d\ni (x,y)\mapsto \Phi(x-y)$ 
is a reproducing kernel for $\mathcal{N}_{\Phi}(\mathbb{R}^d)$. 
If $\widehat{\Phi}$ satisfies 
\begin{equation}
\label{eq:2.3}
 c_1(1+|\xi|^2)^{-\kappa}\le \widehat{\Phi}(\xi)\le c_2(1+|\xi|^2)^{-\kappa}, 
 \quad \xi\in\mathbb{R}^d, 
\end{equation}
for some constants $c_1,c_2>0$ and $\kappa>d/2$, 
then Corollary 10.13 in \cite{wen:2005} yields
$H^{\kappa}(\mathbb{R}^d)= \mathcal{N}_{\Phi}(\mathbb{R}^d)$ and 
\begin{equation}
\label{eq:2.4}
 C_1\|f\|_{H^{\kappa}(\mathbb{R}^d)}\le \|f\|_{\mathcal{N}_{\Phi}(\mathbb{R}^d)}
 \le C_2\|f\|_{H^{\kappa}(\mathbb{R}^d)}, \quad f\in H^{\kappa}(\mathbb{R}^d),
\end{equation}
for constants $C_1,C_2>0$.  Thus the native space
$\mathcal{N}_{\Phi}(\mathbb{R}^d)$ coincides with the Sobolev space
$H^{\kappa}(\mathbb{R}^d)$ with equivalent norms.  Furthermore,
\eqref{eq:2.4} and Corollary 10.25 in \cite{wen:2005} imply
\begin{equation}
\label{eq:2.5}
 \|I(g)\|_{H^{\kappa}(\mathbb{R}^d)}\le C\|g\|_{H^{\kappa}(\mathbb{R}^d)}, \quad 
 \|g-I(g)\|_{H^{\kappa}(\mathbb{R}^d)}\le C\|g\|_{H^{\kappa}(\mathbb{R}^d)}, 
 \quad g\in H^{\kappa}(\mathbb{R}^d). 
\end{equation}

The Wendland kernels provide standard examples satisfying
\eqref{eq:2.3}--\eqref{eq:2.5}.  For a given $\tau\in\mathbb{N}$, define
$\Phi_{d,\tau}$ by
$\Phi_{d,\tau}(x)=\phi_{d,\tau}(|x|)$, $x\in\mathbb{R}^d$, where 
\begin{equation*}
 \phi_{d,\tau}(r)=\int_r^{\infty}r_{\tau}\int_{r_{\tau}}^{\infty}r_{\tau-1}\int_{r_{\tau-1}}^{\infty}
 \cdots\: r_2\int_{r_2}^{\infty}r_1\max\{1-r_1,0\}^{\nu}dr_1dr_2\cdots dr_{\tau}, \quad 
 r\ge 0 
\end{equation*}
with $\nu =\lfloor d/2\rfloor +\tau+1$. 
For example, 
\begin{align*}
 \phi_{1,2}(r)&\doteq \max\{1-r,0\}^5(8r^2+5r+1), \\
 \phi_{1,3}(r)&\doteq \max\{1-r,0\}^7(21r^3 + 19r^2 + 7r + 1),  \\
 \phi_{1,4}(r) &\doteq  \max\{1-r,0\}^9(384r^4 + 453r^3 + 237r^2 + 63r + 7),  \\
 \phi_{2,4}(r) &\doteq \max\{1-r,0\}^{10}(429r^4+450r^3+210r^2+50r+5), \\ 
 \phi_{2,5}(r) &\doteq 
 \max\{1-r,0\}^{12}(2048r^5 + 2697r^4 + 1644r^3 + 566r^2 + 108r + 9),  
\end{align*}
where $\doteq$ denotes equality up to a positive constant factor. 

Then $\Phi_{d,\tau}\in C^{2\tau}(\mathbb{R}^d)$ and
$\mathcal{N}_{\Phi_{d,\tau}}(\mathbb{R}^d)=H^{\tau+(d+1)/2}(\mathbb{R}^d)$. 
Furthermore, $\Phi_{d,\tau}$ satisfies \eqref{eq:2.3}--\eqref{eq:2.5} with 
$\kappa=\tau+(d+1)/2$.

\section{A collocation method for Zakai equations}\label{sec:3}

We now describe the collocation method for \eqref{eq:1.1}.
In what follows, we work with the version $\tilde{u}$ constructed in
Lemma \ref{lem:smooth-version} and, by a slight abuse of notation, write it
simply as $u$.
Moreover, we restrict ourselves to the class of Wendland kernels 
$\Phi=\Phi_{d,\tau}$ described in Section \ref{sec:2.3}. 
We choose pairwise distinct points $x_1,\ldots,x_N$ such that
\begin{equation*}
 \Gamma=\{x_1,\ldots,x_N\}\subset (-R,R)^d, 
\end{equation*}
for some $R>0$. 

To construct an approximate solution of the Zakai equation,
we first choose a time grid $\{t_0,\ldots,t_n\}$ such that
$0=t_0<t_1<\cdots<t_n=T$. 
The solution $u$ of the Zakai equation approximately satisfies 
\begin{equation*}
 u(t_i,x)\simeq u(t_{i-1},x) + \sum_{k=0}^mL_ku(t_{i-1},x)\Delta W_k(t_i), 
\end{equation*}
where
\[
 W_0(t)=t,
 \qquad
 \Delta W_k(t_i)=W_k(t_i)-W_k(t_{i-1}).
\]
Approximating $L_ku(t_{i-1},x)$ by
$L_kI(u(t_{i-1},\cdot))(x)$ gives
\begin{equation*}
 u(t_i,x)\simeq u(t_{i-1},x) + \sum_{k=0}^m L_kI(u(t_{i-1},\cdot))(x)\Delta W_k(t_i).  
\end{equation*}
We therefore define an approximate solution $u^h$, parametrized by $h>0$, by
\begin{equation}
\label{eq:3.1}
\begin{aligned}
 u^{h}(t_0,x)&=u_0(x), \quad x\in\mathbb{R}^d, \\
 u^{h}(t_i,x)&=u^{h}(t_{i-1},x)+\sum_{k=0}^mL_k
  I(u^h(t_{i-1},\cdot))(x)
  \Delta W_k(t_i), \quad x\in\mathbb{R}^d, \;\; i=1,\ldots,n. 
\end{aligned}
\end{equation}
With $u_{i,j}^h=u^h(t_i,x_j)$, the vector
$u_i^{h}=(u_{i,1}^{h},\ldots,u_{i,N}^{h})^{\mathsf{T}}$ of nodal values satisfies
\begin{equation*}
\begin{aligned}
 u_0^{h}&=(u_0(x_1),\ldots,u_0(x_N))^{\mathsf{T}}, \\ 
 u_i^{h} &= u_{i-1}^{h} + \sum_{k=0}^m(A_kA^{-1}u_{i-1}^{h}) \Delta W_k(t_i), 
\quad i=1,\ldots,n. 
\end{aligned}
\end{equation*}
Here, $A_k=(A_{k,j\ell})_{1\le j,\ell\le N}$, where
$A_{k,j\ell}=L_k\Phi(x-x_{\ell})|_{x=x_j}$. 
This follows from  
\begin{equation*}
 L_kI(u^{h}(t_i,\cdot))(x_j)=\sum_{\ell=1}^N(A^{-1}u_i^{h})_{\ell}L_k\Phi(x-x_{\ell})|_{x=x_j} 
 =(A_kA^{-1}u_i^{h})_j. 
\end{equation*}

To analyze the approximation error, set
$\Delta t=\max_{1\le i\le n}(t_i-t_{i-1})$ and consider 
the fill distance $\Delta x$ of $\Gamma$ in $(-R,R)^d$, defined by
\begin{equation*}
 \Delta x = \sup_{x\in (-R,R)^d}\min_{j=1,\ldots,N}|x-x_j|. 
\end{equation*}
We regard $\Delta t$, $R$, $N$, and $\Delta x$ as functions of $h$.
Let $\tilde{\Gamma}$ be a finite subset of $(-R,R)^d$ at which the
approximate solution is evaluated; this set may differ from $\Gamma$.
For $j=1,\ldots,N$, we write $Q_j$ for the cardinal function defined by 
\[
 Q_j(x) = \sum_{i=1}^N(A^{-1})_{ij}\Phi(x-x_i), \quad x\in\mathbb{R}^d, \;\; j=1,\ldots,N. 
\] 
For $k=0,1,\ldots,m$, set
\begin{equation}
\label{eq:3.stab}
 M_{k,h}=\max_{x\in\Gamma\cup\tilde{\Gamma}}
 \sum_{j=1}^N|L_kQ_j(x)|,
 \qquad
 \kappa_h=1+2M_{0,h}+2\sum_{k=1}^mM_{k,h}^2
 +2\Delta t\,M_{0,h}^2,
\end{equation}
and put $\varepsilon_h=\Delta t+(\Delta x)^{2\tau-3}$.
\begin{assum}
\label{assum:3.1}
The parameters $\Delta t$, $R$, $N$, and $\Delta x$ satisfy 
$\Delta t\to 0$, $R\to \infty$, $N\to\infty$, and $\Delta x\to 0$ as $h\searrow 0$. 
\end{assum}


\begin{rem}
\label{rem:3.2}
The quantities in \eqref{eq:3.stab} are finite-dimensional and can be evaluated
directly for a prescribed point set and kernel. The error bound below does not assume
that they are uniformly bounded. Convergence follows if
\begin{equation}
\label{eq:3.balance}
 \exp(T\kappa_h)\varepsilon_h\longrightarrow0.
\end{equation}
A stronger sufficient condition is
$\kappa_h=o(|\log\varepsilon_h|)$. Thus the stability factors may grow, but their
finite-time amplification must be dominated by the consistency error.
\end{rem}

The approximation error for the Zakai equation satisfies the following bound.
\begin{thm}
\label{thm:3.2}
Suppose that Assumptions \ref{assum:2.1} and \ref{assum:3.1} hold and that
$\tau\ge 3$.
Then, there exists $h_0>0$ such that 
\begin{equation*}
\max_{i=1,\ldots,n}\max_{x\in \Gamma\cup\tilde{\Gamma}}\mathbb{E}\left[|u(t_i,x)- u^{h}(t_i,x)|^2
 \right] \le C\exp(T\kappa_h)
 \left(\Delta t+ (\Delta x)^{2\tau-3}\right), \quad h\le h_0.    
\end{equation*}
Consequently, if \eqref{eq:3.balance} holds, then the left-hand side converges to zero
as $h\searrow0$.
\end{thm}

The estimate does not require the stability factors $M_{k,h}$ to be uniformly
bounded.  Assumption \ref{assum:3.1} alone, however, does not imply
\eqref{eq:3.balance}.  In particular, the theorem establishes convergence for
a prescribed kernel and point-set sequence only when the corresponding
stability factors satisfy that asymptotic condition.

%

The rest of this section is devoted to the proof of Theorem \ref{thm:3.2}. 
We begin with the following lemma.
\begin{lem}
\label{lem:3.5}
Suppose that Assumption \ref{assum:3.1} holds and that $\tau\ge 3$. Then
there exists $h_0>0$ such that, for every multi-index $\alpha$ with
$|\alpha|_1\le 2$ and every
$f\in H^{\tau+(d+1)/2}(\mathbb{R}^d)$,
\begin{equation*}
 \|D^{\alpha}f-D^{\alpha}I(f)\|_{L^{\infty}(-R,R)^d}\le 
 C(\Delta x)^{\tau+1/2-|\alpha|_1}\|f\|_{H^{\tau+(d+1)/2}(\mathbb{R}^d)}, 
 \quad h\le h_0. 
\end{equation*}
\end{lem}
\begin{proof}
This is a special case of \cite[Corollary 11.33]{wen:2005}, which applies to
more general domains.
Because the domains below expand with $R$, we make the domain dependence explicit and
confirm that $C$ can be taken independently of $R$.

Let $f\in H^{\tau+(d+1)/2}(\mathbb{R}^d)$ with $f|_{\Gamma}=0$. 
Put $s_*=\tau+(d+1)/2$ and write $s_*=k+\sigma$, where
$k=\lceil s_*\rceil-1\in\mathbb{N}$ and $\sigma\in(0,1]$. For
$|\alpha|_1\le2$ and $\tau\ge3$, we have $k>|\alpha|_1+d/2$.
When $\sigma=1$, the seminorm below is the integer-order seminorm defined in
Section \ref{sec:2.1}; when $0<\sigma<1$, it is the Slobodeckij seminorm.
Set $\tilde{\Gamma}=\{x_1/R,\ldots,x_N/R\}$ and 
$\tilde{f}(z)=f(Rz)$, $z\in (-1,1)^d$. 
Then, $\tilde{f}|_{\tilde{\Gamma}}=f|_{\Gamma}=0$ and 
\begin{equation*}
 \sup_{z\in (-1,1)^d}\min_{\xi\in\tilde{\Gamma}}|\xi-z|
 =\sup_{y\in (-R,R)^d}\min_{j=1,\ldots,N}\left|\frac{x_j}{R}-\frac{y}{R}\right|
 =\frac{\Delta x}{R}. 
\end{equation*}
The cube $(-1,1)^d$ has a fixed interior cone condition. Moreover,
$\Delta x/R\to 0$ as $h\searrow0$. Thus the fill-distance threshold in
\cite[Theorem 11.32]{wen:2005} is satisfied for all sufficiently small $h$.
Applying that theorem with $p=2$, $q=\infty$, smoothness $k+\sigma=s_*$,
and derivative order $|\alpha|_1$, we obtain
\begin{equation}
\label{eq:3.2}
 |D^{\alpha}\tilde{f}(z)|\le C(\Delta x/R)^{\tau+1/2-|\alpha|_1}
 |\tilde{f}|_{H^{\tau+(d+1)/2}((-1,1)^d)}, \quad h\le h_0
\end{equation}
for some $h_0>0$. 
The scaling identities are
\[
 D^{\alpha}\tilde{f}(z)=R^{|\alpha|_1}(D^{\alpha}f)(Rz), \quad 
 |\tilde{f}|_{H^{\tau+(d+1)/2}((-1,1)^d)}=R^{\tau+1/2}|f|_{H^{\tau+(d+1)/2}((-R,R)^d)}. 
\]
Substitution into \eqref{eq:3.2} yields
\begin{equation}
\label{eq:3.3}
 |D^{\alpha}f(y)|\le C(\Delta x)^{\tau+1/2-|\alpha|_1}|f|_{H^{\tau+(d+1)/2}((-R,R)^d)}, 
 \quad y\in (-R,R)^d. 
\end{equation}
Applying \eqref{eq:3.3} to $f-I(f)$, which vanishes on $\Gamma$, and then using
\eqref{eq:2.5}, we obtain
\begin{align*}
 \|D^{\alpha}f-D^{\alpha}I(f)\|_{L^{\infty}((-R,R)^d)}
 &\le C(\Delta x)^{\tau+1/2-|\alpha|_1}|f-I(f)|_{H^{\tau+(d+1)/2}((-R,R)^d)} \\ 
 &\le C(\Delta x)^{\tau+1/2-|\alpha|_1}\|f\|_{H^{\tau+(d+1)/2}(\mathbb{R}^d)}. 
\end{align*}
Thus the lemma follows. 
\end{proof}

Observe that for any $f:\mathbb{R}^d\to\mathbb{R}$, 
\begin{equation*}
 I(f)(x)=\sum_{j=1}^N(A^{-1}f|_{\Gamma})_j\Phi(x-x_j)
 = \sum_{j=1}^Nf(x_j)Q_j(x), \quad x\in\mathbb{R}^d.  
\end{equation*}
The following elementary estimate records precisely the stability factor used below.
\begin{lem}
\label{lem:3.7}
Let $f_j$, $j=1,\ldots,N$, be square-integrable random variables and set
$I_hf=\sum_{j=1}^N f_jQ_j$. Then, for $k=0,1,\ldots,m$,
\begin{equation*}
 \max_{x\in\Gamma\cup\tilde{\Gamma}}
 \mathbb{E}|L_kI_hf(x)|^2
 \le M_{k,h}^2\max_{j=1,\ldots,N}\mathbb{E}|f_j|^2.
\end{equation*}
\end{lem}
\begin{proof}
For $x\in\Gamma\cup\tilde{\Gamma}$, the Cauchy--Schwarz inequality gives
\begin{align*}
 \mathbb{E}|L_kI_hf(x)|^2
 &\le \sum_{j,\ell=1}^N
 (\mathbb{E}|f_j|^2)^{1/2}(\mathbb{E}|f_\ell|^2)^{1/2}
 |L_kQ_j(x)||L_kQ_\ell(x)| \\
 &\le \max_{j=1,\ldots,N}\mathbb{E}|f_j|^2
 \left(\sum_{j=1}^N|L_kQ_j(x)|\right)^2.
\end{align*}
Taking the maximum over $x$ proves the assertion.
\end{proof}

\begin{proof}[Proof of Theorem $\ref{thm:3.2}$]
First, for $i=0,\ldots,n-1$ and $x\in\Gamma\cup\tilde{\Gamma}$, we have 
\begin{align*}
 (u(t_{i+1},x)-u^{h}(t_{i+1},x))^2&=(u(t_i,x)-u^{h}(t_i,x))^2+ 
 (S_{i+1}(x))^2 +(\Theta_{i+1}(x))^2 \\
 &\quad + 2(u(t_i,x)-u^{h}(t_i,x))S_{i+1}(x) 
 +2\Theta_{i+1}(x)S_{i+1}(x) \\ 
 &\quad + 2(u(t_i,x)-u^{h}(t_i,x))\Theta_{i+1}(x). 
\end{align*}
Here, for $i=0,\ldots,n-1$ and $x\in\Gamma\cup\tilde{\Gamma}$, 
\begin{align*}
 S_{i+1}(x) &= \sum_{k=0}^mL_kI(u(t_i)-u^{h}(t_i))(x)\Delta W_{t_{i+1}}^k, \\
 \Theta_{i+1}(x)&=\sum_{k=0}^m\int_{t_i}^{t_{i+1}}(L_ku(s,x)-L_kI(u(t_i))(x))dW_s^k. 
\end{align*}
Put
\[
 E_i=\max_{y\in\Gamma\cup\tilde{\Gamma}}
 \mathbb{E}|u(t_i,y)-u^h(t_i,y)|^2,
 \qquad q_i=t_{i+1}-t_i.
\]
By the independence of the Wiener increments and Lemma \ref{lem:3.7},
\begin{align*}
 \mathbb{E}|S_{i+1}(x)|^2
 &=q_i^2\mathbb{E}|L_0I(u(t_i)-u^h(t_i))(x)|^2
 +q_i\sum_{k=1}^m\mathbb{E}|L_kI(u(t_i)-u^h(t_i))(x)|^2 \\
 &\le \left(q_i^2M_{0,h}^2+q_i\sum_{k=1}^mM_{k,h}^2\right)E_i.
\end{align*}
Hence,
\begin{equation}
\label{eq:3.6}
 \mathbb{E}|S_{i+1}(x)|^2
 \le q_i\left(\Delta t\,M_{0,h}^2+\sum_{k=1}^mM_{k,h}^2\right)E_i.
\end{equation}

Next, It{\^o}'s isometry gives
\begin{align*}
 &\mathbb{E}|\Theta_{i+1}(x)|^2 \\ 
 &\le 2\mathbb{E}\left|\int_{t_i}^{t_{i+1}}
  (L_0u(s,x)-L_0I(u(t_i))(x))ds\right|^2 
  + 2\mathbb{E}\left|\sum_{k=1}^m\int_{t_i}^{t_{i+1}}
  (L_ku(s,x)-L_kI(u(t_i))(x))dW_s^k\right|^2 \\ 
 &\le 2\Delta t\mathbb{E}\int_{t_i}^{t_{i+1}}|L_0u(s,x)-L_0I(u(t_i))(x)|^2ds 
  + 2\sum_{k=1}^m\mathbb{E}\int_{t_i}^{t_{i+1}}|L_ku(s,x)-L_kI(u(t_i))(x)|^2ds \\ 
 &\le C\sum_{|\alpha|_1\le 2}\mathbb{E}\int_{t_i}^{t_{i+1}}
  |D^{\alpha}u(s,x)-D^{\alpha}I(u(t_i))(x)|^2ds. 
\end{align*}
By the Cauchy--Schwarz inequality for the drift integral, It{\^o}'s isometry for
the stochastic integrals, and Sobolev embedding, 
\begin{align*}
 \mathbb{E}|D^{\alpha}u(s,x)-D^{\alpha}u(t_i,x)|^2
 &=\mathbb{E}\left|\sum_{k=0}^m\int_{t_i}^sD^{\alpha}L_ku(r,x)dW_r^k\right|^2 
  \le C\sum_{k=0}^m\mathbb{E}\int_{t_i}^s|D^{\alpha}L_ku(r,x)|^2dr \\
  &\le C(s-t_i), 
\end{align*}
where $W_0(t)=t$ and the constant is finite by the Sobolev estimates in
Section \ref{sec:2.2}, taken at an integer order larger than $d/2+4$.
Furthermore, Lemma \ref{lem:3.5} gives
\begin{align*}
 \mathbb{E}|D^{\alpha}u(t_i,x)-D^{\alpha}I(u(t_i))(x)|^2
 \le C(\Delta x)^{2\tau-3}
 \mathbb{E}\sup_{0\le t\le T}\|u(t)\|^2_{H^{\tau+(d+1)/2}(\mathbb{R}^d)}. 
\end{align*}
Thus, 
\begin{equation}
\label{eq:3.7}
 \mathbb{E}|\Theta_{i+1}(x)|^2
 \le Cq_i\left(\Delta t+(\Delta x)^{2\tau-3}\right). 
\end{equation}
The preceding estimates also yield
\begin{equation}
\label{eq:3.8}
\begin{aligned}
 &\mathbb{E}(u(t_i,x)-u^h(t_i,x))S_{i+1}(x) \\ 
 &=\mathbb{E}(u(t_i,x)-u^h(t_i,x))L_0I(u(t_i)-u^h(t_i))(x)(t_{i+1}-t_i) \\ 
 &\le (\mathbb{E}|u(t_i,x)-u^h(t_i,x)|^2)^{1/2} 
     (\mathbb{E}|L_0I(u(t_i)-u^h(t_i))(x)|^2)^{1/2}q_i \\
 &\le q_iM_{0,h}E_i. 
\end{aligned}
\end{equation}
Here, we used Lemma \ref{lem:3.7} in the last inequality.
Furthermore, $u(t_i,x)-u^h(t_i,x)$ is $\mathcal F_{t_i}$-measurable, so the
stochastic-integral part of $\Theta_{i+1}(x)$ has zero cross expectation with it. Hence
we obtain
\begin{equation*}
\begin{aligned}
 &\mathbb{E}(u(t_i,x)-u^h(t_i,x))\Theta_{i+1}(x) \\
 &=\mathbb{E}(u(t_i,x)-u^h(t_i,x))
  \int_{t_i}^{t_{i+1}}(L_0u(s,x)-L_0I(u(t_i))(x))ds \\ 
 &\le (\mathbb{E}|u(t_i,x)-u^h(t_i,x)|^2)^{1/2} 
  \left(\mathbb{E}\left|\int_{t_i}^{t_{i+1}}(L_0u(s,x)
    -L_0I(u(t_i))(x))ds\right|^2\right)^{1/2} \\ 
 &\le (\mathbb{E}|u(t_i,x)-u^h(t_i,x)|^2)^{1/2}
   \left(q_i\mathbb{E}\int_{t_i}^{t_{i+1}}\left|L_0u(s,x)
    -L_0I(u(t_i))(x)\right|^2ds\right)^{1/2} \\ 
 &=(\mathbb{E}|u(t_i,x)-u^h(t_i,x)|^2q_i)^{1/2}
   \left(\mathbb{E}\int_{t_i}^{t_{i+1}}\left|L_0u(s,x)
    -L_0I(u(t_i))(x)\right|^2ds\right)^{1/2} \\ 
 &\le \frac12\mathbb{E}|u(t_i,x)-u^h(t_i,x)|^2q_i 
    + \frac12\mathbb{E}\int_{t_i}^{t_{i+1}}\left|L_0u(s,x)
    -L_0I(u(t_i))(x)\right|^2ds. 
\end{aligned}
\end{equation*}
Therefore,
\begin{equation}
\label{eq:3.9}
\begin{aligned}
 &2\left|\mathbb{E}(u(t_i,x)-u^h(t_i,x))\Theta_{i+1}(x)\right| \\
 &\le q_iE_i+Cq_i\left(\Delta t+(\Delta x)^{2\tau-3}\right). 
\end{aligned}
\end{equation}

Moreover, $2|ab|\le |a|^2+|b|^2$ gives
\begin{equation*}
 2\mathbb{E}|\Theta_{i+1}(x)S_{i+1}(x)|
 \le \mathbb{E}|\Theta_{i+1}(x)|^2+\mathbb{E}|S_{i+1}(x)|^2.
\end{equation*}
Combining \eqref{eq:3.6}--\eqref{eq:3.9}, we obtain, for $i=0,\ldots,n-1$,
\begin{align*}
 E_{i+1}
 &\le (1+\kappa_hq_i)E_i
  +Cq_i\left(\Delta t+(\Delta x)^{2\tau-3}\right).
\end{align*}
Since $E_0=0$, the discrete Gronwall lemma yields
\[
 \max_{i=1,\ldots,n}E_i
 \le C\exp(T\kappa_h)
 \left(\Delta t+(\Delta x)^{2\tau-3}\right).
\]
Finally, \eqref{eq:3.balance} gives the convergence assertion.
\end{proof}

\section{A finite-resolution illustration}\label{sec:4}

We briefly illustrate the finite-resolution performance of the collocation
method for the one-dimensional Zakai equation
\begin{equation}
\label{eq:4.1}
\left\{
\begin{aligned}
 du(t,x)&=\left(\frac{1}{2}\frac{\partial^2}{\partial x^2} u(t,x) 
  - \frac{\partial}{\partial x}(\tanh(x)u(t,x))\right)dt 
  + u(t,x)dW(t), \quad 0\le t\le 1, \\
 u(0,x)&=\frac{1}{\sqrt{2\pi}}\cosh (x)e^{-|x|^2/2}. 
\end{aligned}
\right.
\end{equation}
The unique solution $u(t,x)$ to \eqref{eq:4.1} is given by 
\begin{equation*}
 u(t,x)=\frac{1}{\sqrt{2\pi(1+t)}}\cosh(x)
  \exp\left(W(t)-t - \frac{|x|^2}{2(1+t)}\right). 
\end{equation*}

We use the Wendland kernel $\phi_{1,4}$ with unit support radius, a uniform time grid in
$[0,1]$, and $N$ uniform collocation points on
$[-R+2R/(N+1),R-2R/(N+1)]$, where
$R=(1/5)N^{1-1/(2\tau-2)}$.  Errors are evaluated at the 41 equally spaced
points $\{\xi_1,\ldots,\xi_{41}\}$ in $[-2,2]$.  At these off-grid points we
report the postprocessed interpolant
\begin{equation*}
 \bar u^h(t_i,x):=I(u^h(t_i,\cdot))(x)
 =\sum_{j=1}^Nu^h(t_i,x_j)Q_j(x).
\end{equation*}
This is the usual reconstruction from the computed nodal vector.  It is not
the recursively defined off-grid value $u^h(t_i,x)$ in \eqref{eq:3.1}, and
Theorem \ref{thm:3.2} is not asserted here as an error bound for this
postprocessed quantity.  Since the noise operator in \eqref{eq:4.1} is the
identity, the first and second moments needed below satisfy closed deterministic
recursions.  We therefore compute the expected RMSE directly, without Monte
Carlo sampling error:
\begin{align*}
\text{RMSE} &:= \sqrt{\frac{1}{41(n+1)}\sum_{i=0}^n\sum_{j=1}^{41}
 \mathbb E|u(t_i,\xi_j)-\bar u^{h}(t_i,\xi_j)|^2}.
\end{align*} 
The resulting finite-resolution errors are shown in Table \ref{table:4.1}.
\begin{table}[htb]
\centering
\begin{tabular}[t]{ccccc} 
\toprule
$N$ & $R$ & $n=2^6$ & $n=2^8$ & $n=2^{10}$ \\ \midrule
$2^4$ & 2.0159 & 0.0867 & 0.0852 & 0.0849 \\
$2^5$ & 3.5919 & 0.0216 & 0.0108 & 0.0056 \\
$2^6$ & 6.4000 & 0.0216 & 0.0108 & 0.0054 \\ \bottomrule
\end{tabular}
\caption{RMSE for the tested finite resolutions.}
\label{table:4.1}
\end{table}

The errors decrease as the time grid is refined at the tested resolutions,
although their decrease is not strictly monotone in $N$.  This table is only a
performance illustration of the postprocessed interpolant.
It neither verifies the asymptotic balance condition \eqref{eq:3.balance} nor
establishes convergence of this particular kernel and point-set sequence.  The
observed accuracy may reflect cancellations or dissipativity that are not
captured by the absolute row-sum estimate in Theorem \ref{thm:3.2}.

\subsection*{Acknowledgements}

This work was partially supported by JSPS KAKENHI Grant Number JP17K05359.

\clearpage

\newcommand{\noop}[1]{}

\end{document}